\documentclass[12pt,a4paper]{amsart}
\usepackage[utf8]{inputenc}
\usepackage{mathrsfs, amssymb}
\usepackage[left=2cm,right=2cm,top=2cm,bottom=2cm]{geometry}

\author{Adam Rennie}
\author{David Robertson}
\author{Aidan Sims}

\email{renniea@uow.edu.au, dave84robertson@gmail.com, asims@uow.edu.au}
\address{School of Mathematics and Applied Statistics\\
University of Wollongong\\
Wollongong\\ NSW\\ 2522\\
AUSTRALIA}


\newtheorem{theorem}{Theorem}[section]
\newtheorem{lemma}[theorem]{Lemma}
\newtheorem{proposition}[theorem]{Proposition}
\newtheorem{corollary}[theorem]{Corollary}
\theoremstyle{definition}
\newtheorem{definition}[theorem]{Definition}
\theoremstyle{remark}
\newtheorem{remark}[theorem]{Remark}
\newtheorem{example}[theorem]{Example}

\numberwithin{equation}{section}

\newcommand{\N}{\mathbb{N}}
\newcommand{\Z}{\mathbb{Z}}

\newcommand{\inv}{^{-1}}

\newcommand{\Gg}{\mathcal{G}}
\newcommand{\Hh}{\mathcal{H}}
\newcommand{\Kk}{\mathcal{K}}
\newcommand{\Ll}{\mathcal{L}}
\newcommand{\Tt}{\mathcal{T}}
\newcommand{\Oo}{\mathcal{O}}
\newcommand{\Zz}{\mathcal{Z}}
\newcommand{\NT}{\mathcal{NT}}
\newcommand{\NO}{\mathcal{NO}}
\newcommand{\Ee}{\mathscr{E}}
\newcommand{\Omax}{\Omega_{\text{max}}}
\DeclareMathOperator{\Her}{Her}

\DeclareMathOperator{\opsp}{span}
\DeclareMathOperator{\clsp}{\overline{\opsp}}

\begin{document}
\title[Product systems and Fell bundles]{Groupoid Fell bundles for product systems over quasi-lattice ordered groups}

\begin{abstract}
Consider a product system over the positive cone of a quasi-lattice ordered group. We
construct a Fell bundle over an associated groupoid so that the cross-sectional algebra
of the bundle is isomorphic to the Nica--Toeplitz algebra of the product system. Under the
additional hypothesis that the left actions in the product system are implemented by
injective homomorphisms, we show that the cross-sectional algebra of the restriction of
the bundle to a natural boundary subgroupoid coincides with the Cuntz--Nica--Pimsner
algebra of the product system. We apply these results to improve on existing sufficient
conditions for nuclearity of the Nica--Toeplitz algebra and the Cuntz--Nica--Pimsner
algebra, and for the Cuntz--Nica--Pimsner algebra to coincide with its co-universal
quotient.
\end{abstract}

\maketitle

\section{Introduction}

In \cite{Pimsner1997}, Pimsner associated to each $C^*$-correspondence over a
$C^*$-algebra $A$ two $C^*$-algebras $\Tt_X$ and $\Oo_X$. His construction simultaneously
generalised the Cuntz--Krieger algebras and their Toeplitz extensions, graph
$C^*$-algebras and crossed products by $\mathbb{Z}$, and has been intensively studied
ever since.

It is standard these days to present $\Tt_X$ as the universal $C^*$-algebra generated by
a representation of the module $X$, and then $\Oo_X$ as the quotient of $\Tt_X$
determined by a natural covariance condition. However, this was not Pimsner's original
definition. In \cite{Pimsner1997}, $\Oo_X$ is by definition the quotient of the image of
the canonical representation of $X$ as creation operators on its Fock space by the ideal
of compact operators on the Fock space. Pimsner then provided two alternative
presentations of $\Oo_X$, the second of which is the one in terms of its universal
property. The first, which is the one germane to this paper, is an analogue of the
realisation of $C(\mathbb{T})$ by dilation of the canonical representation of the
classical Toeplitz algebra on $\ell^2$. Pimsner constructed a direct-limit module
$X_\infty$ over the direct limit $A_\infty$ of the algebras of compact operators on the
tensor powers of $X$. He showed that one can make sense of $X^{\otimes n}_\infty$ for all
integers $n$, and so form a 2-sided Fock space $\bigoplus_{n \in \Z} X^{\otimes
n}_\infty$. This space carries a natural representation of $X_\infty$ by translation
operators, and the image generates $\Oo_{X_\infty}$ which is isomorphic to $\Oo_X$.

More recently \cite{Fowler2002}, Fowler introduced compactly aligned product systems of
Hilbert $A$--$A$ bimodules over the positive cones in quasi-lattice ordered groups
$(G,P)$, and studied associated $C^*$-algebras $\Tt_X$ and $\Oo_X$, and an interpolating
quotient $\NT_X$ (Fowler denoted it by $\Tt_{\operatorname{cov}}(X)$, but we follow the
notation of \cite{BaHLR}). When $(G, P) = (\Z, \N)$, $\Tt_X = \NT_X$ agrees with
Pimsner's Toeplitz algebra, and $\Oo_X$ with Pimsner's Cuntz-Pimsner algebra. But even
for $(\Z^2, \N^2)$ the situation is more complicated. The algebra $\NT_X$ is essentially
universal for the relations encoded by the natural Fock representation of $X$, so it is a
natural analogue of Pimsner's Toeplitz algebra. But the quotient by the ideal of compact
operators on the Fock space is much too large to behave like an analogue of Pimsner's
$\Oo_X$. (This is analogous to the fact that $C^*(\Z)$ is the quotient of $C^*(\N)$ by
the compact operators on $\ell^2(\N)$, but $C^*(\Z^2)$ is much smaller than the quotient
of $C^*(\N^2)$ by the compact operators on $\ell^2(\N^2)$.) Fowler also lacked an
analogue of $X_\infty$; the direct limit should be taken over $P$, but $P$ is typically
not directed. So Fowler's approach to defining $\Oo_X$ was to mimic Pimsner's second
alternative presentation of $\Oo_X$: identify a natural covariance relation and
\emph{define} $\Oo_X$ as the universal quotient of $\Tt_X$ determined by this relation.
Subsequent papers \cite{SY, CLSV2009} have modified Fowler's definition to accommodate
various levels of additional generality, but have taken the same fundamental approach of
defining $\NO_X$ as the universal $C^*$-algebra determined by a representation of $\Tt_X$
satisfying some additional essentially ad hoc relations. Nevertheless, there is strong
evidence \cite{Fowler2002, CLSV2009} that the resulting $C^*$-algebra $\NO_X$ can
profitably be regarded as a generalised crossed product of the coefficient algebra $A$ by
the group $G$. In particular, in the case that $(G, P) = (\Z^k, \N^k)$ and $X$ is the
product system arising from an action $\alpha$ of $\N^k$ on $A$ by endomorphisms, a new
characterisation and analysis of $\NO_X$, closely related to Pimsner's dilation approach,
is achieved in \cite{DFK} using the powerful machinery of Arveson envelopes of
non-self-adjoint operator algebras. The authors answer in the affirmative a question
raised in \cite{CLSV2009} about whether $\NO_X$ can be recovered using Arveson's
approach, and use this to show, amongst other things, that $\NO_X$ is Morita equivalent
(in fact isomorphic in the case that the $\alpha_p$ are all injective) to a genuine
crossed-product by $\Z^k$.

In this paper we provide an analogue of Pimsner's first representation of $\Oo_X$ that is
applicable to compactly aligned product systems over quasi-lattice ordered groups, under
the additional hypothesis that the left $A$-actions are implemented by nondegenerate
injective homomorphisms $\phi_p : A \to \Ll(X_p)$. Our approach is to use a natural
groupoid $\Gg$ associated to $(G, P)$ \cite{MR1982}, and construct a Fell bundle over
$\Gg$ whose cross-sectional $C^*$-algebra coincides with $\NT_X$. The groupoid $\Gg$ has
a natural boundary, which is a closed subgroupoid (see \cite{CL}), and the restriction of
our Fell bundle to this boundary subgroupoid has cross-sectional algebra isomorphic to
the algebra $\NO_X$ of \cite{SY}. This is strong evidence that the relations recorded in
\cite{SY} are the right ones, at least for nondegenerate product systems with injective
left actions. As practical upshots of our results, we deduce that if the groupoid $\Gg$
is amenable, then: (1) each of $\NT_X$ and $\NO_X$ is nuclear whenever the coefficient
algebra $A$ is nuclear, and (2) $\NO_X$ coincides with its co-universal quotient
$\NO^r_X$ as in \cite{CLSV2009}. This improves on previous results along these lines,
which assume that the group $G$ is amenable, a stronger hypothesis than amenability of
$\Gg$.

We mention that the work of Kwasniewski and Szyma\'nski in \cite{KS}, is related to our
construction. There the authors consider product systems over semigroups $P$ that satisfy
the Ore condition but are not necessarily part of a quasi-lattice ordered pair, and
assume that the left actions in the product system are by compact operators. Here, by
contrast, we insist that $P$ is quasi-lattice ordered, but do not require compact
actions. Both approaches use the machinery of Fell bundles, but Kwasniewski and
Szyma\'nski construct Fell bundles over the enveloping group $G$ of $P$, whereas we
construct a bundle over the associated groupoid $\Gg$; as mentioned above, an advantage
of the latter is that $\Gg$ can be amenable even when $G$ is not.

\section{Preliminaries}

\subsection{Product systems over quasi-lattice ordered groups}

Let $G$ be a discrete group and let $P$ be a subsemigroup of $G$ satisfying $P\cap P\inv
= \{ e \}$. Define a partial order $\leq$ on $G$ by
\[
 g \leq h \Longleftrightarrow g\inv h \in P.
\]
We call the pair $(G,P)$ a \emph{quasi-lattice ordered group} if, whenever two elements
$g,h \in G$ have a common upper bound in $G$, they have a least common upper bound $g\vee
h$ in $G$. We write $g\vee h < \infty$ if two elements $g,h \in G$ have a common upper
bound and $g\vee h = \infty$ otherwise.

A \emph{product system} over a quasi-lattice ordered group $(G,P)$ is a semigroup $X$
equipped with a semigroup homomorphism $d : X \to P$ such that the following hold. For
each $p\in P$, let $X_p = d\inv(p)$. Then we require that $A = X_e$ is a $C^*$-algebra,
thought of as a right-Hilbert module over itself in the usual way, and that each $X_p$ is
a right-Hilbert $A$-module together with a left action of $A$ by adjointable operators
denoted $\varphi_p : A \to \Ll(X_p)$. We require that $\varphi_e$ is given by left
multiplication. Furthermore, for each $p,q \in P$ with $p\neq e$, we require that
multiplication in $X$ determines a Hilbert bimodule isomorphism $X_p\otimes_A X_q \to
X_{pq}$ satisfying $x_p\otimes x_q \mapsto x_p x_q$. The product system is
\emph{nondegenerate} if multiplication $X_e \times X_p \to X_p$ also determines an
isomorphism $X_e \otimes_a X_p \to X_p$ for each $p$; that is, if each $X_p$ is
nondegenerate as a left $A$-module. Every right Hilbert module is automatically
nondegenerate as a right $A$-module by the Hewitt-Cohen factorisation theorem.

If $p,q \in P$ satisfy $e\neq p\leq q$, then there is a homomorphism $i_{p\inv q} :
\Ll(X_p) \to \Ll(X_q)$ characterised by
\[
 i_{p\inv q}(S)(xy) = (Sx)y \ \textrm{ for all } x\in X_p,\ y\in X_{p\inv q}.
\]
If we identify $A$ with $\Kk(X_e)$ in the usual way then the corresponding map $i_p :
\Kk(X_e) \to \Ll(X_p)$ is $i_p = \varphi_p$. We say that a product system $X$ is
\emph{compactly aligned} if, whenever $S\in\Kk(X_p), T \in \Kk(X_q)$ and $p\vee q <
\infty$ we have
\[
 i_{p\inv(p\vee q)}(S)i_{q\inv(p\vee q)}(T) \in \Kk(X_{p\vee q}).
\]
If $g\in G\setminus P$ we define $i_g$ to be $0$.

\begin{example}
The pair $(\mathbb{Z},\mathbb{N})$ is a quasi-lattice ordered group, where $\leq$ agrees
with the usual ordering on $\mathbb{Z}$. Let $A$ be a $C^*$-algebra and let $E$ be an
$A$-correspondence; i.e. $E$ is a right Hilbert $A$-module with a left action $A \to
\Ll(E)$. Let $X_0:= A$ and for each $n \in \mathbb{N} \setminus \{0\}$ let $X_n :=
E^{\otimes n}$. Then
\[\textstyle
 X := \bigcup_{n\in\mathbb{N}} X_n
\]
is a product system over $(\mathbb{Z},\mathbb{N})$. With multiplication given by $\xi\eta
:= \xi \otimes \eta$.
\end{example}

\begin{example}
For each $k\geq 1$, the pair $(\mathbb{Z}^k,\mathbb{N}^k)$ is a quasi lattice ordered
group where, for $m,n\in\mathbb{Z}^k$ and $1\leq i \leq k$
\[
 (m\vee n)_i = \max\{m_i,n_i\}.
\]
Suppose that $(\Lambda,d)$ is a $k$-graph. For each $n\in\mathbb{N}^k$, $C_c(d\inv(n))$
is a pre-Hilbert $A = C_0(\Lambda^0)$ module. Let $X_n = \overline{C_c(d\inv(n))}$. Then
\[\textstyle
 X =  \bigcup_{n\in\mathbb{N}^k} X_n
\]
is a product system over $(\mathbb{Z}^k,\mathbb{N}^k)$. (See \cite{RS}.)
\end{example}

\subsection{Representations of product systems}

For details of the following, see \cite{CLSV2009, Fowler2002, SY}.

\begin{definition}
\label{def:nica-cov} Let $X$ be a compactly aligned product system over a quasi-lattice
ordered group $(G,P)$. A \emph{Toeplitz representation} of $X$ in a $C^*$-algebra $B$ is
a map $\psi : X \to B$ satisfying
\begin{enumerate}
 \item[(T1)]\label{T1} $\psi_p := \psi|_{X_p} : X_p \to B$ is linear for all $p\in P$
     and $\psi_e$ is a homomorphism,
 \item[(T2)]\label{T2} $\psi(xy) = \psi(x)\psi(y)$ for all $x,y \in X$, and
 \item[(T3)]\label{T3} for any $p\in P$ and $x,y\in X_p$, $\psi(\langle x,y \rangle)
     = \psi(x)^*\psi(y)$.
\end{enumerate}

Given a Toeplitz respresentation $\psi : X \to B$, for each $p\in P$ there is a
homomorphism $\psi^{(p)} : \Kk(X_p) \to B$ satisfying
\[
 \psi^{(p)}(\theta_{x,y}) = \psi_p(x)\psi_p(y)^*.
\]
We call a Toeplitz representation $\psi : X \to B$ \emph{Nica covariant} if
\begin{enumerate}
 \item[(N)]\label{N} for all $S\in\Kk(X_p), T\in\Kk(X_q)$ we have
 \[
 \psi^{(p)}(S)\psi^{(q)}(T) = \left\{ \begin{array}{ll}
 \psi^{p\vee q}\left( i_{p\inv(p\vee q)}(S)i_{q\inv(p\vee q)}(T) \right) & \textrm{ if } p\vee q< \infty \\
 0 & \textrm{ otherwise}.
 \end{array} \right.
 \]
\end{enumerate}
Following \cite{BaHLR}, we will write $\NT_X$ for the universal $C^*$-algebra generated
by a Nica-covariant Toeplitz representation $i_X$ of $X$. (Fowler shows that such a
$C^*$-algebra exists in \cite{Fowler2002}, but denotes it $\Tt_{\operatorname{cov}}(X)$.)
\end{definition}

Given a predicate $\mathcal{P}$ on $P$, we say $\mathcal{P}$ is true \emph{for large s}
if for every $q\in P$, there exists an $r\geq q$ such that $\mathcal{P}(s)$ is true
whenever $s\geq r$.

We now present the definition of the Cuntz--Nica--Pimsner algebra $\NO_X$ of a product
system $X$ under the assumption that the left action on each fibre is implemented by an
injective homomorphism $\varphi_p$. This hypothesis is not needed for $\NO_X$ to make
sense (see \cite{SY}); but if the left actions are not implemented by injective
homomorphisms, then the relation~(CNP) as described below does not hold in $\NO_X$. In
particular, this hypothesis will be necessary in all statements that involve
Cuntz-Nica-Pimsner covariance and representations of $\NO_X$:
Proposition~\ref{prop:cp_covariant_representation}, Theorem~\ref{thm:CNPisomorphism}, and
the results in Section~\ref{sec:applications}

\begin{definition}
\label{def:c-n-p-cov} Let $X$ be a compactly aligned product system over a quasi-lattice
ordered group $(G,P)$ and suppose that for each $p\in P$ the left action $\phi_p : A \to
\Ll(X_p)$ is injective. We say a Nica covariant Toeplitz representation $\psi : X \to B$
is \emph{Cuntz--Nica--Pimsner covariant} if it satisfies the following property:
\begin{enumerate}
\item[(CNP)]\label{CNP} for each finite $F \subset P$ and collection of elements $T_p
    \in \Kk(X_p)$, $p \in F$,
    \[\textstyle
        \text{if }\sum_{p\in F} i_{p\inv q}(T_p) = 0\text{ for large $q$,}\qquad\text{ then }\qquad
            \sum_{p\in F} \psi^{(p)}(T_p) = 0.
    \]
\end{enumerate}
We write $\NO_X$ for the universal $C^*$-algebra generated by a Cuntz--Nica--Pimsner
covariant representation $j_X$ of $X$.
\end{definition}

\subsection{Fell bundles over groupoids} We say that a groupoid $\Gg$ is a
\emph{topological groupoid} if $\Gg$ is a topological space and the multiplication and
inversion are continuous functions. We call a topological groupoid $\Gg$ \emph{\'etale}
if the unit space $\Gg^{(0)}$ is locally compact and Hausdorff, and the range map $r :
\Gg \to \Gg^{(0)}$ is a local homeomorphism. It follows that the source map $s$ is also a
local homeomorphism. A \emph{bisection} of $G$ is an open subset $U \subseteq G$ such
that $r|_U$ and $s|_U$ are homeomorphisms; the topology of a Hausdorff \'etale groupoid
admits a basis consisting of bisections. See \cite{Exel} for an overview of \'etale
groupoids.

Given a Hausdorff \'etale groupoid $\Gg$, a \emph{Fell bundle} over $\Gg$ is an
upper-semicontinuous Banach bundle $p : \Ee\to\Gg$ with a multiplication
\[
 \Ee^{(2)} = \{ (e,f) \in \Ee\times \Ee : (p(e),p(f))\in\Gg^{(2)} \} \to \Ee
\]
and an involution
\[
 * : \Ee \to \Ee, \ e\mapsto e^*
\]
satisfying the following properties:
\begin{enumerate}
\item\label{it:FB1} the multiplication is associative and bilinear, whenever it makes
    sense;
\item $p(ef) = p(e)p(f)$ for all $(e,f)\in \Ee^{(2)}$;
\item multiplication is continuous in the relative topology on $\Ee^{(2)} \subseteq
    \Ee\times \Ee$;
\item $\|ef\| \leq \|e\| \|f\|$ for all $(e,f)\in \Ee^{(2)}$;
\item $p(e^*) = p(e)^{-1}$  for all $e\in \Ee$, and involution is continuous and
    conjugate linear;
\item $(e^*)^* = e, \|e^*\| = \|e\|$ and $(ef)^* = f^* e^*$ for all $(e,f)\in
    \Ee^{(2)}$;
\item\label{it:FB7} $\|e^* e\| = \|e\|^2$ for all $e \in \Ee$;
\item $e^*e \geq 0$ as an element of $p^{-1}(s(p(e)))$---which is a $C^*$-algebra by
    (\ref{it:FB1})--(\ref{it:FB7})---for all $e\in \Ee$.
\end{enumerate}
We denote by $E_\gamma$ the fibre $p\inv(\gamma) \subset \Ee$.

Given a Fell bundle $\Ee$ over a locally compact Hausdorff \'etale groupoid, we write
$\Gamma_c(\Gg;\Ee)$ for the vector space of continuous, compactly supported sections $\xi
: \Gg \to \Ee$. If $\Hh \subseteq \Gg$ is a closed subset, we will write $\Gamma_c(\Hh;
\Ee)$ for the compactly supported sections of the restriction of $\Ee$ to $\Hh$; that is,
$\Gamma_c(\Hh; \Ee) := \Gamma_c(\Hh; \Ee|_\Hh)$.

There are a convolution and involution on $\Gamma_c(\Gg;\Ee)$ such that for $\xi,\eta \in
\Gamma_c(\Gg;\Ee)$,
\[
 (\xi * \eta)(\gamma) = \displaystyle\sum_{\alpha\beta = \gamma} \xi(\alpha)\eta(\beta)
 \quad\text{ and }\quad
 \xi^*(\gamma) = \xi(\gamma\inv)^*.
\]
This gives $\Gamma_c(\Gg;\Ee)$ the structure of a $*$-algebra. The \emph{$I$-norm} on
$\Gamma_c(\Gg;\Ee)$ is given by
\[
 \|f\|_I := \sup_{u\in\Gg^{(0)}}\Big( \max\Big( \sum_{s(\gamma)=u} \|f(\gamma)\|, \sum_{r(\gamma)=u} \|f(\gamma)\| \Big) \Big).
\]
A $*$-homomorphism $L : \Gamma_c(\Gg;\Ee) \to \mathcal{B}(\mathcal{H}_L)$ is called a
\emph{bounded representation} if $\|L(f)\| \leq \|f\|_I$ for all $f\in\Gamma_c(\Gg;\Ee)$.
It is \emph{nondegenerate} if $\clsp\{ L(f)\xi : f\in \Gamma_c(\Gg;\Ee), \xi \in
\mathcal{H}_L \} = \mathcal{H}_L$ is dense. The \emph{universal $C^*$-norm} on
$\Gamma_c(\Gg;\Ee)$ is
\[
 \|f\| := \sup \{ \|L(f)\| : L \textrm{ is an bounded representation} \}.
\]
We define the \emph{cross-sectional algebra} $C^*(\Gg,\Ee)$ to be the completion of
$\Gamma_c(\Gg;\Ee)$ with respect to the universal $C^*$-norm. If $\Hh \subseteq \Gg$ is a
closed subgroupoid, then we write $C^*(\Hh, \Ee)$ for the completion of $\Gamma_c(\Hh,
\Ee)$ in the universal norm on $\Gamma_c(\Hh, \Ee)$.

\section{From a product system to a Fell bundle}\label{sec:Fell bundle}

In this section, given a product system $X$ over a quasi-lattice ordered group $(G,P)$,
we construct a groupoid $\Gg$ and a Fell bundle $\Ee$ over $\Gg$. We will show in
Section~\ref{sec:isomorphisms} that the $C^*$-algebra of this Fell bundle coincides with
the Nica--Toeplitz algebra of $X$, and has a natural quotient that coincides with the
Cuntz--Nica--Pimsner algebra.

\textbf{Standing notation:} We fix, for the duration of Section~\ref{sec:Fell bundle}, a
quasi-lattice ordered group $(G,P)$, and a nondegenerate compactly aligned product system
$X$ over $P$. For the time being, we do not require that the left actions on the fibres
of $X$ are implemented by injective homomorphisms; as mentioned before, this additional
hypothesis will be needed only in Proposition~\ref{prop:cp_covariant_representation},
Theorem~\ref{thm:CNPisomorphism}, and the results of Section~\ref{sec:applications}.

\subsection{The groupoid}

We first construct a groupoid from $(G, P)$. This construction is by no means new---for
example, it appears in the work of Muhly and Renault \cite{MR1982} in the context of
Weiner-Hopf algebras. Fix a quasi-lattice ordered group $(G,P)$. We say that $\omega
\subset G$ is \emph{directed} if
\[
 g,h \in \omega \implies \infty\neq g\vee h \in \omega
\]
and \emph{hereditary} if
\[
 h\in \omega\text{ and }g\leq h \implies g\in\omega.
\]
Let $\Omega = \{\omega \subset G : \omega \textrm{ is directed and hereditary}\}$. With
the relative product topology induced by identifying $\Omega$ with a subset of
$\{0,1\}^G$ in the usual way, $\Omega$ is a totally disconnected compact Hausdorff space:
the sets
\[
 \Zz(A_0,A_1) := \{\omega \in \Omega : g\in A_i \implies \chi_\omega (g) = i \},
\]
indexed by pairs $A_0, A_1$ of finite subsets of $G$ constitute a basis of compact open
sets.

We say that $\omega\in\Omega$ is \emph{maximal} if $\omega \subset \rho \in \Omega$
implies $\omega = \rho$. Let $\Omax = \{\omega\in \Omega : \omega \text{ is maximal}\}$.
Define the \emph{boundary} of $\Omega$ to be
\[
 \partial\Omega := \overline{\Omax} \subset \Omega.
\]

Given $g\in G$ and $\omega \in \Omega$, let
\[
 g\omega := \{gh:h\in\omega\}.
\]
For finite $A_0, A_1 \subseteq G$ and $g \in G$, we have $g\inv \Zz(A_0,A_1) = \Zz(g\inv
A_0,g\inv A_1)$. Hence $g\cdot \omega := g\omega$ defines an action of $G$ by
homeomorphisms of $\Omega$. Given $p\in P$, the set $\omega_p := \{g\in G : g\leq p\}$
belongs to $\Omega$, so we can regard $P$ as a subset of $\Omega$.

\begin{proposition}
The boundary $\partial\Omega$ is invariant under the action of $G$.
\end{proposition}

\begin{proof} By continuity of the $G$-action, it suffices to show that $\Omax$ is invariant.
Fix $\omega\in\Omax$ and $g\in G$ and suppose that $g\omega \subset \rho$ for some
$\rho\in\Omega$. Then $\omega \subset g\inv\rho$ and hence $\omega = g\inv\rho$, since
$\omega$ is maximal. So $g\omega = gg\inv\rho = \rho$.
\end{proof}

The set
\[
 \Gg = \{(g,\omega) : P\cap\omega \neq\varnothing\text{ and } P\cap g\omega \neq\varnothing \}
\]
becomes a groupoid when endowed with the operations
\[
 (g,h\omega)(h,\omega) = (gh,\omega)\quad\text{ and }\quad
 (g,\omega)\inv = (g\inv,g\omega).
\]
The unit space is $\{e\} \times \Omega$, which we identify with $\Omega$, and the structure
maps are
\[
r(g,\omega) = (e,g\omega) \quad\textrm{ and }\quad s(g,\omega) = (e,\omega).
\]
One can check that $\Gg$ is equal to the restriction of the transformation groupoid $G
\ltimes \Omega$ to the closure of the copy of $P$ in $\Omega$; in symbols, $\Gg =
(G\ltimes \Omega)|_{\overline{P}}$. We write $\Gg|_{\partial\Omega}$ for the subgroupoid
\[
 \Gg|_{\partial\Omega} := \{(g,\omega) \in \Gg : \omega \in \partial\Omega\}.
\]

\subsection{The fibres of the Fell bundle}
\label{subsec:fell-fibres}

For a fixed $r\in P$ and any $p,q\in P$ there is a map
\[
 i_r : \Ll(X_p,X_q) \to \Ll(X_{pr},X_{qr})
\]
such that, for $x\in X_p$ and $y\in X_r$
\[
 i_r(S)(xy) = S(x)y.
\]
There is no notational dependence on $p$ and $q$, but this will not cause
confusion---indeed, it is helpful to think of $i_r$ as a map from $\bigoplus_{p, q \in P}
\Ll(X_p, X_q)$ to $\bigoplus_{p, q \in P} \Ll(X_{pr}, X_{qr})$.

For $\omega \in \Omega$ and $p \in \omega$, we define $[p,\omega) := \{q \in \omega : p
\le q\}$. Given any $(g,\omega) \in \Gg$, we have $[e \vee g^{-1}, \omega) = \{p \in P
\cap \omega : gp \in P\}$, and this set is directed (under the usual ordering on $P$). So
we can form the Banach-space direct limit
\[\textstyle
 \varinjlim_{p \in [e \vee g^{-1}, \omega)} \Ll(X_p,X_{gp})
\]
with respect to the maps $i_r : \Ll(X_p,X_{gp}) \to \Ll(X_{pr},X_{gpr})$ where $pr, gpr
\in \omega$. By definition of the direct limit, there are bounded linear maps
$\Ll(X_p,X_{gp}) \to \varinjlim \Ll(X_p,X_{gp})$, $p \in [e \vee g^{-1}, \omega)$, that
are compatible with the linking maps $i_r$. To lighten notation we regard all of these
maps as components of a single map $i_{(g,\omega)} : \bigoplus_p \Ll(X_p, X_{gp}) \to
\varinjlim \Ll(X_p,X_{gp})$. We define
\[\textstyle
 E_{(g,\omega)} := \clsp{\bigcup_{p \in [e \vee g^{-1}, \omega)} i_{(g,\omega)} (\Kk(X_p,X_{gp}))}
 \subset \varinjlim \Ll(X_p,X_{gp}).
\]

\begin{lemma}\label{lem:fibres}
Each $A_\omega := E_{(e,\omega)}$ is a $C^*$-algebra and each $E_{(g,\omega)}$ is an
$A_{g\omega}$--$A_\omega$ imprimitivity bimodule.
\end{lemma}
\begin{proof}
By definition of the maps $i_r$, if $T \in \Ll(X_p, X_{p'})$ and $S \in \Ll(X_{p'},
X_{p''})$, then $i_r(T)i_r(S) = i_r(TS)$, and $i_r(T)^* = i_r(T^*)$. Using this, one
checks that, identifying each $\Ll(X_p \oplus X_{gp})$ with the algebra of block-operator
matrices $\Big(\begin{smallmatrix}\Ll(X_p) & \Ll(X_{gp}, X_p)\\ \Ll(X_p, X_{gp}) &
\Ll(X_{gp})\end{smallmatrix}\Big)$, the maps $i_r$ determine a homomorphism $i_r :
\Ll(X_p \oplus X_{gp}) \to \Ll(X_{pr} \oplus X_{gpr})$. In the same vein as above, we use
the notation $\tilde{\imath}_{g,\omega}$ for all of the homomorphisms $\Ll(X_p \oplus
X_{gp}) \to \varinjlim \Ll(X_p, X_{gp})$.

The following is adapted from the proof of \cite[Lemma~4.1]{LR}. Since $\omega$ is
directed, each finite subset $H \subseteq [e \vee g^{-1}, \omega)$ is contained in a
finite $F \subseteq [e \vee g^{-1}, \omega)$ which is closed under $\vee$, and each such
$F$ has a maximum element $p_F$. For each such $F$, let
$$
B_F := \sum_{s \in F}
i_{s^{-1}p_F}(\Kk(X_s \oplus \Kk(X_{gs})) \subseteq \Ll(X_{p_F} \oplus X_{gp_F}).
$$
If $F
\subseteq \omega$ is finite with more than one element and $\vee$-closed, and if $q \in
F$ is minimal, then $F' := F \setminus \{q\}$ is also $\vee$ closed, and $p_{F'} = p_F$.
We have $B_F = i_{q^{-1}p_F}(\Kk(X_q \oplus X_{gq})) + B_{F'}$. Nica covariance and
minimality of $q$ ensures that
\[
i_{q^{-1}p_F}(\Kk(X_q \oplus X_{gq}))i_{s^{-1}p_F}(\Kk(X_s \oplus X_{gs}))
    \subseteq i_{(q \vee s)^{-1}p_F}(\Kk(X_(q \vee s) \oplus X_{g(q \vee s)}))
    \subseteq B_{F'}
\]
So $B_{F'} B_F, B_F B_{F'} \subseteq B_{F'}$. Assuming as an inductive hypothesis that
$B_{F'}$ is a $C^*$-algebra, we deduce from \cite[Corollary~1.8.4]{Dixmier} that $B_F$ is
a $C^*$-algebra. Since each $B_{\{p\}} = \Kk(X_p \oplus X_{gp})$ is clearly a
$C^*$-algebra, we conclude by induction that each $B_F$ is a $C^*$-algebra. So
\[\textstyle
\clsp{\bigcup_{p \in [e \vee g^{-1}, \omega)} \tilde{\imath}_{g,\omega} (\Kk(X_p \oplus X_{gp}))}
    \subset \varinjlim \Ll(X_p \oplus X_{gp})
\]
is canonically isometrically isomorphic to $L_{g,\omega} := \varinjlim_F
\tilde{\imath}_{g,\omega}(B_F)$, so is a $C^*$-algebra. Put $p = e \vee g^{-1}$, so $p
\in \omega \cap P$ and $gp \in g\omega \cap P$. Since $X$ is nondegenerate, the spaces
$A_\omega$ and $A_{g\omega}$ appear as the complementary full corners
$\tilde{\imath}_{g,\omega}(1_{X_p}) L_{g,\omega} \tilde{\imath}_{g,\omega}(1_{X_p})$ and
$\tilde{\imath}_{g,\omega}(1_{X_{gp}}) L_{g,\omega}
\tilde{\imath}_{g,\omega}(1_{X_{gp}})$ of $L_{g,\omega}$, so they are $C^*$-algebras.
Furthermore, $E_{(g,\omega)} = \tilde{\imath}_{g,\omega}(1_{X_{gp}}) L_{g,\omega}
\tilde{\imath}_{g,\omega}(1_{X_{p}})$, and so it is an
$A_{g\omega}$--$A_{\omega}$-imprimitivity bimodule.
\end{proof}

\subsection{The operations on the Fell bundle}

Let
\[\textstyle
 \Ee := \bigcup_{(g,\omega) \in \Gg} E_{(g,\omega)}.
\]
Then $\Ee$ is a bundle over $\Gg$, with $\pi : \Ee \to \Gg$ defined by
$\pi(E_{(g,\omega)}) = \{(g,\omega)\}$.

\begin{lemma}\label{lemma:compactly_aligned}
Fix $p,p',q,q' \in P$ with $p\vee q' < \infty$ and let $r = p{}\inv(p\vee q')$, and $r' =
q'{}\inv(p\vee q')$. Then for any $S \in \Kk(X_{p},X_{p'})$ and $T \in \Kk(X_q,X_{q'})$
we have
\[
 i_r(S)i_{r'}(T) \in \Kk(X_{q r'}, X_{p' r}).
\]
\end{lemma}
\begin{proof}
Since both the left and right actions are nondegenerate, it is enough to prove the result
for $SU$ and $VT$ where $S\in \Kk(X_{p,p'}), U\in \Kk(X_{p})$ and $T\in \Kk(X_q,X_{q'}),
V \in \Kk(X_{q'})$. We have
 \[
  i_r(SU) i_{r'}(VT) = i_r(S) i_r(U)i_{r'}(V) i_{r'}(T).
 \]
Since $X$ is compactly aligned, we have $i_r(U)i_{r'}(V) \in \Kk(X_{p\vee q'})$, and
hence $i_r(SU) i_{r'}(VT) \in \Kk(X_{q r'}, X_{p' r})$ as claimed.
\end{proof}

Fix $((g,h\omega),(h,\omega)) \in \Gg^{(2)}$, $hp\in [e\vee g\inv,h\omega), q\in
[e\vee h\inv, \omega)$ and $S\in \Kk(X_{hp},X_{ghp})$, $T\in \Kk(X_q,X_{hq})$. Let $r = p\inv
(p\vee q), r' = q\inv (p\vee q)$, and define
\[
 i_{(g,h\omega)}(S)  i_{(h,\omega)}(T) := i_{(gh,\omega)}\left( i_r(S)i_{r'}(T) \right).
\]
The right hand side makes sense by Lemma \ref{lemma:compactly_aligned}. This extends to a
multiplication
\[
 \Ee^{(2)} := \{(e,f) \in \Ee \times \Ee : (\pi(e),\pi(f)) \in \Gg^{(2)} \} \to \Ee.
\]
For $(g,\omega)\in \Gg$ and $p\in [e \vee g^{-1}, \omega)$, the usual adjoint operation
$* :\Ll(X_p,X_{gp}) \to \Ll(X_{gp},X_p) = \Ll(X_{gp},X_{g\inv (gp)})$ is isometric. So
for each $(g,\omega)$ it extends to an involution $\varinjlim \Ll(X_p,X_{gp}) \to
\varinjlim \Ll(X_{gp},X_p)$, which then restricts to an involution $E_{(g,\omega)} \to
E_{(g\inv,g\omega)}$.

\subsection{The topology on the Fell bundle}
Given $p,q\in P$ and $S\in\Ll(X_p,X_q)$ define $f^S : \Gg \to \bigcup_{(g,\omega) \in
\Gg} \varinjlim_{p \in [e \vee g^{-1}, \omega)} \Ll(X_p, X_{gp})$ by
\[
 f^S(g,\omega) = \left\{
 \begin{array}{ll}
 i_{(qp\inv,\omega)}(S) & \textrm{ if } g = qp\inv \textrm{ and } p\in\omega \\
 0 & \textrm{ otherwise.}
 \end{array}
 \right.
\]

\begin{lemma}\label{lem:usc}
For any $p, q \in P$ and any $S\in\Ll(X_p,X_q)$, the map
\[
 (g,\omega) \mapsto \|f^S(g,\omega)\|
\]
is upper semicontinuous.
\end{lemma}
\begin{proof}
Since $\| f^S(g,\omega)\| = \|f^{S^*S}(\omega) \|^{1/2}$ for any $(g,\omega) \in \Gg$, it
is enough to check upper semicontinuity on the unit space $\Gg^{(0)} = \Omega$. Fix $p
\in P$,  $S \in \Ll(X_p)$ and $\alpha > 0$. We must show that the set
\[
 \{ \omega : \|f^S(\omega)\| < \alpha \}
\]
is open. Since $p\not\in\omega$ implies that $f^S(\omega) = 0$, we see that
\[
 \{ \omega : \|f^S(\omega)\| < \alpha \} = \Zz(\{p\},\varnothing) \cup \{ \omega : p\in\omega \text{ and } \|i_\omega(S)\| < \alpha \}
\]
and so it is enough to show that $\{ \omega : p\in\omega \text{ and } \|f^S(\omega)\| <
\alpha \}$ is open. Fix $\omega$ in this set. Since $A_\omega$ is a direct limit we have
\[
\|f^S(\omega)\| = \|i_\omega(S)\| = \lim_{q\geq p} \|i_{qp\inv}(S)\| = \inf_{q\geq p} \|i_{qp\inv}(S)\|.
\]
Therefore, there exists a $q\geq p$ such that $\|i_{qp\inv}(S)\| < \alpha$. Suppose that
$\omega' \in \Zz(\varnothing,\{q\})$. Then $p\in\omega'$, and so
\[
 \|f^S(\omega')\| = \|i_{\omega'}(S)\| \leq \|i_{qp\inv}(S)\| <\alpha.\qedhere
\]
\end{proof}

Now let
\[
\Gamma = \opsp\{ f^S :  p,q \in P,\ S\in \Kk(X_p,X_q) \}.
\]
Given finitely many pairs $(p_1, q_1), \dots, (p_n, q_n)$ and operators $S_i \in
\Kk(X_{p_i}, X_{q_i})$, there are finitely many maximal subsets $F_1, \dots, F_m$ of
$\{p_1, \dots, p_n\}$ such that each $F_j$ has an upper bound $r_j$ in $P$. Putting $T_j
:= \sum_{p \in F_j} i_{p^{-1}r_j}(S_i)$ for each $j$, we have $T_j \in \Ll(X_{r_j})$ and
\[\textstyle
\sum_{i=1}^n f^{S_i}
    = \sum^m_{j=1} f^{T_j},
\]
where the $f^{T_j}$ have mutually disjoint support. So Lemma~\ref{lem:usc} shows that the
sections in $\Gamma$ are upper semicontinuous.

Given $(g,\omega) \in \Gg$ we have
\begin{align*}
 \{ f(g,\omega) : f\in \Gamma \} &= \left\{ i_{(g,\omega)}(S) :\ p\in [e \vee g^{-1}, \omega),\ S \in \Kk(X_p,X_{gp}) \right\} \\
 &=\textstyle \bigcup_{[e \vee g^{-1}, \omega)} i_{(g,\omega)} (\Kk(X_p,X_{gp}))
\end{align*}
which densely spans $E_{(g,\omega)}$. Hence \cite[Section II.13.18]{FD1988} shows that
there is a unique topology on $\Ee$ such that $(\Ee,\pi)$ is a Banach bundle and all the
functions in $\Gamma$ are continuous cross sections of $\Ee$; and $\Ee$ becomes a
Fell-bundle over $\Gg$ in this topology.

\section{Representing the product system}

\subsection{Toeplitz representation}
Let $(G,P)$ be a quasi-lattice ordered group, and $X$ a nondegenerate compactly aligned
product system over $P$. For $p\in P$, identify $X_p$ with $\Kk(X_e, X_p)$ as usual: $x
\in X_p$ is identified with the operator $a \mapsto x\cdot a$. We then write $x^*$ for
the operator $y \mapsto \langle x,y\rangle_{X_e}$ in $\Kk(X_p, X_e)$. Define $\psi_p :
X_p \to C^*(\Gg,\Ee)$ by $\psi_p(x) = f^x$.

\begin{proposition}
\label{prop:nica_covariant_representation} Let $(G,P)$ be a quasi-lattice ordered group,
and $X$ a nondegenerate compactly aligned product system over $P$. Let $\Gg$ and $\Ee$ be
the groupoid and Fell bundle constructed in Section~\ref{sec:Fell bundle}. The map $\psi
: X \to C^*(\Gg,\Ee)$ such that $\psi|_{X_p} = \psi_p$ is a Nica covariant Toeplitz
representation of $X$, and for $S \in \Kk(X_p)$, we have $\psi^{(p)}(S) = f^S$.
\end{proposition}
\begin{proof}
We need to check the conditions of Definition~\ref{def:nica-cov}. For $x,y \in X_p$ and
$a\in X_e$,
\begin{align*}
\psi_p(x)^*\psi_p(y)(g,\omega)
    &= [{(f^x}^*) * f^y](g,\omega)
     = \sum_{h\omega \cap P \neq \varnothing} f^x((gh\inv,h\omega)\inv)^* f^y(h,\omega)\\
    &= \sum_{h\omega \cap P \neq \varnothing} f^x(hg\inv,g\omega)^* f^y(h,\omega)
     = \delta_{g,e} f^x(p,\omega)^* f^y(p,\omega) \\
    &= \delta_{g,e} i_{(p,\omega)}(x)^* i_{(p,\omega)}(y)
     = \delta_{g,e} i_{(p\inv,p\omega)}(x^*) i_{(p,\omega)}(y)\\
    &= \delta_{g,e} i_{\omega}(\langle x,y\rangle_A)
     = f^{\langle x,y\rangle_A} (g,\omega)
     = \psi_e(\langle x,y \rangle).
\end{align*}
Likewise,
\begin{align*}
[\psi_e(a) \psi_p(x)](g,\omega)
    &= [f^a * f^x](g,\omega)
     = \sum_{h\omega \cap P \neq \varnothing} f^a(gh\inv,h\omega) f^x(h,\omega) \\
    &= \delta_{g,p} i_{p\omega}(a) i_{(p,\omega)}(x)
     = \delta_{g,p} i_{(p,\omega)}(ax)
     = f^{ax}(g,\omega)
     = \psi_p(ax)
\end{align*}
and
\begin{align*}
[\psi_p(x)\psi_e(a)](g,\omega) &= [f^x * f^a](g,\omega)
     = \sum_{h\omega \cap P \neq \varnothing} f^x(gh\inv,h\omega) f^a(h,\omega) \\
    &= \delta_{g,p} i_{(p,\omega)}(x)i_{\omega}(a)
     = \delta_{g,p} i_{(p,\omega)}(xa)
     = f^{xa}(g,\omega)
     = \psi_p(xa).
\end{align*}
To see that each $\psi^{(p)}(S) = f^S$, consider $S = \theta_{x,y}$ and calculate:
\begin{align*}
\psi^{(p)}(\theta_{x,y})(g,\omega) &= [\psi_p(x)\psi_p(y)^*](g,\omega)
     = [f^x * f^y](g,\omega)
     = \sum_{h\omega \cap P \neq \varnothing} f^x(gh\inv,h\omega) f^y((h,\omega)\inv)^*\\
    &= \sum_{h\omega \cap P \neq \varnothing} f^x(gh\inv,h\omega) f^y(h\inv,h\omega)^*
     = \delta_{g,p} i_{(p,p\inv\omega)}(x) i_{(p,p\inv\omega)}(y)^*\\
    &= \delta_{g,p} i_{(p,p\inv\omega)}(x) i_{(p\inv,\omega)}(y^*)
     = \delta_{g,p} i_\omega(\theta_{x,y})
     = f^{\theta_{x,y}}(g,\omega).
\end{align*}
So continuity and linearity give $\psi^{(p)}(S) = f^S$ for all $S \in \Kk(X_p)$. Fix $p,q
\in P$ with $p\vee q<\infty$ and $S\in\Kk(X_p), T\in\Kk(X_q)$. Then
\begin{align*}
[\psi^{(p)}(S) \psi^{(q)}(T)](g,\omega) &= [f^S * f^T](g,\omega)
     = \sum_{h\omega \cap P \neq \varnothing} f^S(gh\inv,h\omega) f^T(h,\omega) \\
    &= \delta_{g,e} i_\omega(S)i_\omega(T)
     = \delta_{g,e} i_\omega(i_{p\inv(p\vee q)}(S)i_{q\inv(p\vee q)}(T))\\
    &= f^{i_{p\inv(p\vee q)}(S)i_{q\inv(p\vee q)}(T)}(g,\omega)
     = [\psi^{(p\vee q)}(i_{p\inv(p\vee q)}(S)i_{q\inv(p\vee q)}(T))](g,\omega).
\end{align*}
Thus all the conditions of Definition \ref{def:nica-cov} are satisfied.
\end{proof}

\subsection{Restriction of the representation to the boundary groupoid}

Consider $\pi_p : X_p \to C^*(\Gg|_{\partial\Omega}, \Ee)$ satisfying
\[
 \pi_p(x) = f^x|_{\Gg|_{\partial\Omega}}
\]
Define $\pi : X \to C^*(\Gg|_{\partial\Omega}, \Ee)$ by $\pi|_{X_p} = \pi_p$.

\begin{proposition} \label{prop:cp_covariant_representation}
Let $(G,P)$ be a quasi-lattice ordered group, and $X$ a nondegenerate compactly aligned
product system over $P$. Suppose that the homomorphisms $\phi_p : A
\to \Ll(X_p)$ implementing the left actions are all injective.
Let $\Gg$ and $\Ee$ be the groupoid and Fell bundle constructed
in Section~\ref{sec:Fell bundle}. The map $\pi : X \to C^*(\Gg|_{\partial\Omega}, \Ee)$
is a Cuntz--Nica--Pimsner covariant Toeplitz representation.
\end{proposition}

Before we prove this, we need two lemmas.

\begin{lemma}
\label{lemma:maximal_sup} Suppose that $\omega\in\partial\Omega$ and $q\in P$ satisfy
$q\vee p < \infty$ for all $p\in\omega$. Then $q\in\omega$.
\end{lemma}
\begin{proof}
Consider the set
\[
 q\vee \omega := \{ q\vee p : p\in\omega \}
\]
If $q\vee p_1, q\vee p_2 \in q\vee\omega$ we have
\[
 (q\vee p_1) \vee (q\vee p_2) = q\vee (p_1\vee p_2) \in q\vee \omega
\]
since $p_1\vee p_2 \in \omega$. So $q\vee\omega$ is directed. Let $\Her(q\vee\omega)$
denote the hereditary closure $\Her(q\vee\omega) = \{g \in G : g \le p \text{ for some }
p \in q \vee \omega\}$ of $q \vee \omega$. Notice that $q = q\vee e
\in\Her(q\vee\omega)$. For any $p\in\omega$,
\[
 p \leq q\vee p \in q\vee\omega
\]
and hence $p\in\Her(q\vee\omega)$. So $\omega \subset \Her(q\vee\omega)$ and hence
$\omega = \Her(q\vee\omega)$ because $\omega\in\partial\Omega$. So $q\in\omega$.
\end{proof}

\begin{lemma}\label{lemma:limit_omega}
Fix a sequence $(\omega_n)_{n=1}^\infty \subset \Omega$ with $p\in\omega_n$ for all $n$,
and suppose that $\omega_n \to \omega$. Then $p\in\omega$, and for $T \in \Kk(X_p)$,
\[
    i_{\omega_n}(T) \to i_\omega(T) \text{ in $E$ as $n\to \infty$.}
\]
\end{lemma}
\begin{proof}
We know that the set $\Zz(\varnothing,\{p\})$ is closed and $\omega_n\in
\Zz(\varnothing,\{p\})$ for all $n$. Hence $\omega\in \Zz(\varnothing,\{p\})$ and so
$p\in\omega$.

Now, fix $T\in\Kk(X_p)$ and $U\subset \Ee$ open with $i_\omega(T) \in U$. By definition
of the topology on $\Ee$, the function $f^T$ is continuous, so $(f^T)\inv(U) \subset \Gg$
is open. Since $\omega_n \to \omega$ and $\Gg$ has the relative product topology,
$(e,\omega_n) \to (e,\omega)$ in $\Gg$. We have $f^T(e,\omega) = i_\omega(T) \in U$, and
hence $(e,\omega)\in (f^T)\inv(U)$. Thus there exists $N$ such that $(e,\omega_n) \in
(f^T)\inv(U)$ for all $n > N$, and so
\[
    f^T(e,\omega_n) = i_{\omega_n}(T) \in U\text{ for all $n > N$},
\]
giving $i_{\omega_n}(T) \to i_\omega(T)$.
\end{proof}

\begin{proof}[Proof of Proposition \ref{prop:cp_covariant_representation}]
Replacing an $\omega\in\Omega$ with $\omega\in\partial\Omega$ in the proof of Proposition
\ref{prop:nica_covariant_representation} shows that $\pi$ is a Nica covariant Toeplitz
representation. Since all the left actions are by injective homomorphisms, the
representation $\pi$ is Cuntz-Nica-Pimsner covariant if it satisfies relation~(CNP) of
Definition~\ref{def:c-n-p-cov}.

Fix a finite set $F\subset P$ and elements $T_p \in \Kk(X_p)$, $p\in F$ such that
\[\textstyle
 \sum_{p\in F} i_{qp\inv}(T_p) = 0
\]
for large $q$. We must show that $\sum_{p\in F} \pi^{(p)}(T_p) = 0$. So, since each
$\pi^{(p)}(T) = \psi^{(p)}(T)|_{\partial \Omega}$, we have to check that
\[\textstyle
    \sum_{p\in F} f^{T_p}(g,\omega) = 0
\]
for all $(g,\omega)\in\Gg|_{\partial\Omega}$. Fix $(g,\omega)\in\Gg|_{\partial\Omega}$
with $\omega \in \Omax$, and observe that
\[\textstyle
 \sum_{p\in F} f^{T_p}(g,\omega) = \delta_{g,e} \sum_{p\in F\cap\omega} i_\omega(T_p).
\]
Since $F\cap\omega \subset P$ is finite and $\omega$ is directed, the element
\[\textstyle
    r:= \bigvee_{p\in F\cap\omega} p
\]
belongs to $\omega$, and
\[\textstyle
\sum_{p\in F\cap\omega} i_\omega(T_p)
    = i_\omega\Big(\sum_{p\in F\cap\omega} i_{p\inv r}(T_p)\Big).
\]
Since $\omega$ is directed and countable we can choose a sequence $(r_n)_{n=1}^\infty
\subset \omega$ satisfying
\begin{itemize}
 \item $r_1 \geq r$,
 \item $r_{n+1} \geq r_n$ for all $n$
 \item for all $ q\in\omega$, there exists $n$ with $r_n \geq q$.
\end{itemize}
For each $n$, choose $q_n \geq r_n$ and $\omega_n \in \partial\Omega$ with
$q_n\in\omega_n$ (and hence $r_n \in \omega_n$) such that
\[\textstyle
 \sum_{p\in F} i_{p\inv q_n}(T_p) = 0.
\]
Then in particular,
\begin{equation} \label{eqn:qnzero}\textstyle
 \sum_{p\in F\cap\omega_n} i_{p\inv q_n}(T_p) = \sum_{p\in F} i_{p\inv q_n}(T_p) = 0
\end{equation}
since $p\in F\setminus\omega_n$ implies $p\nleq q_n$ and so $i_p^{p\inv q_n} = 0$. We
claim that $\omega_n \to \omega$ as $n\to\infty$. To see this fix $\Zz(A_0,A_1)$
containing $\omega$. Since $A_1 \subset \omega$, $A_1$ is directed. Let
\[\textstyle
 s = \bigvee_{p\in A_1} p.
\]
By definition of $(r_n)_{n=1}^\infty$ there is an $n_1$ with $r_{n_1} \geq s$. Then $A_1
\subset \omega_{r_n}$ for any $n\geq n_1$.

For each $q \in A_0$, let $N_q := \max\{n : q \in  \omega_n\}$. Suppose for contradiction
that $q \in A_0$ satisfies $N_q=\infty$. For any $p \in\omega$ we can find $r_j \geq p$.
Since $N_q = \infty$ we can find $k\geq j$ with $q \in \omega_k$. But then
\[
    q \vee r_k < \infty \implies q \vee r_j < \infty  \implies q \vee p < \infty.
\]
Since $p\in\omega$ was arbitrary we deduce that $q\vee p < \infty$ for all $p\in\omega$
and hence $q\in\omega$ by Lemma~\ref{lemma:maximal_sup}. This contradicts $\omega \in
\Zz(A_0,A_1)$. Therefore $N_q$ is finite for every $q \in A_0$. Now put
\[\textstyle
    N := \max\big\{n_1, \max_{q \in A_0} N_q\big\} < \infty.
\]
Then $\omega_n \in \Zz(A_0,A_1)$ for any $n > N$ and $\omega_n \to \omega$ as claimed.
Since $F$ is finite, there exists $N_F$ such that $n\geq N_F$ implies $F\cap \omega_n =
F\cap \omega$.

Hence, using Lemma~\ref{lemma:limit_omega} at the third equality and~\eqref{eqn:qnzero}
at the last one, we have
\begin{align*}
\sum_{p\in F} f^{T_p} (g,\omega) &= \delta_{g,e} \sum_{p\in F\cap\omega} i_\omega(T_p)
     = \delta_{g,e} i_\omega \left( \sum_{p\in F\cap\omega} i_{p\inv r}(T_p) \right)
     = \delta_{g,e} \lim_{n\to\infty} i_{\omega_n} \left( \sum_{p\in F\cap\omega} i_{p\inv r}(T_p) \right)\\
    &= \delta_{g,e} \lim_{n\to\infty} i_{\omega_n} \left( \sum_{p\in F\cap\omega} i_{p\inv q_n}(T_p) \right)
     = \delta_{g,e} \lim_{n\to\infty} i_{\omega_n} \left( \sum_{p\in F\cap\omega_n} i_{p\inv q_n}(T_p) \right) = 0.
\end{align*}
Since $\Omax$ is dense in $\partial\Omega$ and $\sum_{p \in F} \pi^{(p)}(T_p)$ is a
continuous section of $\Ee$, we deduce that $\sum_{p\in F} \pi^{(p)}(T_p) = 0$.
\end{proof}

\section{The isomorphisms}\label{sec:isomorphisms}

In this section, we prove our main results: that the $C^*$-algebra of the Fell bundle
$\Ee$ constructed in Section~\ref{sec:Fell bundle} is isomorphic to the Nica--Toeplitz
algebra $\NT_X$ and, under the hypothesis that the left actions of $A$ on the $X_p$ are
implemented by injective homomorphisms, that the $C^*$-algebra of the restriction of
$\Ee$ to the boundary groupoid $\Gg|_{\partial \Omega}$ is isomorphic to the
Cuntz--Nica--Pimsner algebra $\NO_X$.

\begin{theorem}
\label{thm:Toeplitzisomorphism} Let $X$ be a compactly aligned product system over a
quasi-lattice ordered group $(G,P)$. Let $\Gg$ and $\Ee$ be the groupoid and Fell bundle
constructed in Section~\ref{sec:Fell bundle}. Then the homomorphism $\Psi : \NT_X \to
C^*(\Gg, \Ee)$ induced by the Toeplitz representation $\psi$ of
Proposition~\ref{prop:nica_covariant_representation} is an isomorphism.
\end{theorem}
\begin{proof}
We begin by showing that $\Psi$ is surjective. By definition of the topology on $\Ee$, it
suffices to show that $f^S \in \operatorname{Im}\Psi$ for all $S\in\Kk(X_p,X_q)$. If $S,
T \in \Kk(X_p,X_q)$ then $f^S + f^T = f^{S+T}$, so it suffices to show that
$f^{\theta_{y,x}} \in \operatorname{Im}\Psi$ for all $x\in X_p$ and $y\in X_q$. Given
$(g,\omega)\in\Omega$ we have
\begin{align*}
 [\psi_q(y)\psi_p(x)^*](g,\omega) &= [f^y * {f^x}^* ](g,\omega)
     = \sum_{h\omega \cap P \neq \varnothing} f^y(gh\inv,h\omega) f^x(h\inv,h\omega)^* \\
    &= \delta_{g,qp\inv} f^y(q,p\inv\omega) f^x(p,p\inv\omega)^*
     = \delta_{g,qp\inv} i_{(q,p\inv\omega)}(x) i_{(p,p\inv\omega)}(y)^*\\
    &= \delta_{g,qp\inv} i_{(q,p\inv\omega)}(x) i_{(p\inv,\omega)}(y^*)
     = \delta_{g,qp\inv} i_{(qp\inv,\omega)}(xy^*)
     = f^{\theta_{x,y}} (g,\omega)
\end{align*}
as required. To see that $\Psi$ is injective, we construct an inverse. We begin by
showing that there is a well-defined map $\Phi : \textrm{span}\{ f^S : S\in \Kk(X_p,X_q)
\} \to \NT_X$ satisfying
\begin{equation}
\label{eqn:well-defined}
 \Phi(f^{\theta_{y,x}}) = i_X(y)i_X(x)^*\quad\text{ for all $x \in X_p$ and $y \in X_q$.}
\end{equation}
To see that such a map exists, suppose that
\[\textstyle
 \sum_{j=1}^n f^{\theta_{y_j,x_j}} = 0 \in \Gamma_c(\Gg;\Ee).
\]
It suffices to show that
\[\textstyle
 \sum_{j=1}^n i_X(y_j)i_X(x_j)^* = 0 \in \NT_X.
\]
Since the Fock representation $l : X \to \Ll(F(X))$ is isometric
\cite[page~340]{Fowler2002}, this is equivalent to
\[\textstyle
 \sum_{j=1}^n l(y_j)l(x_j)^* = 0 \in \Ll(F(X)).
\]
To see this, fix $z\in X_r$ and $a\in A$. For any $p\in P$ we have
\begin{align*}
 \left( \sum_{j=1}^n l(y_j)l(x_j)^*(z\cdot a) \right)(p)
 &= \sum_{\substack{p_j \leq r \\ q_j p_j\inv r = p}} y_j\left( i_{p_j\inv r}(x_j)^*(z\cdot a) \right).
\end{align*}
Hence $\left( \sum_{j=1}^n f^{\theta_{y_j,x_j}} \right) * f^{\theta_{z,a}} = 0$, and so
\begin{align*}
 0  &= \left( \Big( \sum_{j=1}^n f^{\theta_{y_j,x_j}} \Big) * f^{\theta_{z,a}}\right) (p,[e])
     = \sum_{\substack{q_j p_j\inv r = p \\ p_j \in [r] }} i_{(q_jp_j\inv, [r])}(\theta_{y_j,x_j}) i_{(r,[e])}(\theta_{(z,a)}) \\
    &= \sum_{\substack{ p_j\leq r \\ q_jp_j\inv r = p }} i_{p_j\inv r}(\theta_{y_j,x_j}) i_e(\theta_{z,a})
     = \sum_{\substack{p_j \leq r \\ q_j p_j\inv r = p}} y_j\left( i_{p_j\inv r}(x_j)^*(z\cdot a) \right).
\end{align*}
Hence
\[\textstyle
 \left( \sum_{j=1}^n l(y_j)l(x_j)^*(z\cdot a) \right)(p) = 0.
\]
Since $z\cdot a$ and $p$ were arbitrary, we see that there is a well-defined linear map
satisfying~\eqref{eqn:well-defined}.

We now show that $\Phi$ in continuous in the inductive limit topology. Suppose that $f_i
\to f$ in $\Gamma_c(\Gg;\Ee)$. Fix a compact subset $K \subset \Gg$ such that $f$
and each of the $f_i$ vanishes off $K$. Write $f = \sum_{j=1}^n f^{S_j}$ where each $S_j
\in \Kk(X_{p_j},X_{q_j})$. Inductively define
\[\textstyle
A_1 = \operatorname{supp} (f^{S_1}) \qquad\text{ and }\qquad
A_{k+1} = \operatorname{supp}(f^{S_{k+1}}) \setminus \Big(\bigcup_{j=1}^k A_k \Big)
\]
for $1\leq k \leq n$. Then each $A_k\subset \Gg$ is a bisection, so that $\|(f_i -
f)|_{A_k}\|_{C^*(\Gg, \Ee)} = \|(f_i - f)|_{A_k}\|_\infty$ for all $i$. Define the set
\[\textstyle
 A_{n+1} = K \setminus \left( \bigcup_{j=1}^n A_k \right).
\]
Without loss of generality, we may assume that $A_{n+1}$ is also a bisection. Then there
exists $N\geq 1$ such that for all $i\geq N$ and $1 \leq k \leq n$
\[
 \| (f_i - f)|_{A_k} \|_\infty < \frac{\varepsilon}{n+1}.
\]
So for $i\geq N$
\begin{align*}
\|\Phi(f_i) - \Phi(f)\| & = \Big\|\sum_{j=1}^{n}\Phi(f_i - f^{S_j})\Big\|
     = \Big\|\sum_{j=1}^{n} \sum_{k=1}^{n+1} \Phi((f_i - f^{S_j})|_{A_k}) \Big\| \\
    &\leq \sum_{j=1}^{n} \sum_{k=1}^{n+1} \|\Phi((f_i - f^{S_j})|_{A_k}) \|
     \leq \sum_{j=1}^{n} \sum_{k=1}^{n+1} \|(f_i - f^{S_j})|_{A_k} \|_\infty
     < \varepsilon.
\end{align*}
 So $\Phi(f_i) \to \Phi(f)$. Since the inductive limit topology on
 $\Gamma_c(\Gg;\Ee)$ is weaker than the norm topology, we see that
 $\Phi$ is bounded in norm. Since $\Gamma_c(\Gg;\Ee)$ is norm dense in
 $C^*(\Gg,\Ee)$, $\Phi$ extends to a $*$-homomorphism
\[
 \Phi : C^*(\Gg,\Ee) \to \NT_X
\]
which is, by construction, an inverse for $\Psi$. So $C^*(\Gg,\Ee) \cong \NT_X$.
\end{proof}

\begin{theorem}
\label{thm:CNPisomorphism} Let $X$ be a nondegenerate compactly aligned product system
over a quasi-lattice ordered group $(G,P)$. Suppose that the homomorphisms $\phi_p : A
\to \Ll(X_p)$ implementing the left actions are all injective. Let $\Gg$ and $\Ee$ be the
groupoid and Fell bundle constructed in Section~\ref{sec:Fell bundle}. Then the
homomorphism $\Pi : \NO_X \to C^*(\Gg|_{\partial\Omega},\Ee)$, induced by the
Cuntz--Nica--Pimsner covariant representation $\pi$ of
Proposition~\ref{prop:cp_covariant_representation}, is an isomorphism.
\end{theorem}

Before we prove Theorem~\ref{thm:CNPisomorphism}, we need to do some background work on
coactions. The first lemma that we need is a general statement about coactions of
discrete groups. The following brief summary of discrete coactions is based on \cite[\S
A.3]{EKQR}. Given a discrete group $G$, the universal property of $C^*(G)$ shows that
there is a homomorphism $\delta_G : C^*(G) \to C^*(G) \otimes C^*(G)$ whose extension to
$\mathcal{M} C^*(G)$ satisfies $\delta_g(i_G(g)) = i_G(g) \otimes i_G(g)$. A coaction of
a discrete group $G$ on a $C^*$-algebra $A$ is a nondegenerate homomorphism $\delta : A
\to A \otimes C^*(G)$ which satisfies the coaction identity
\[
(\delta \otimes 1_{C^*(G)}) \circ \delta = (1 \otimes \delta_G) \circ \delta.
\]
The coaction $\delta$ is coaction-nondegenerate if $\clsp \delta(A)(1_{\mathcal{M}(A)}
\otimes C^*(G)) = A \otimes C^*(G)$.

It is claimed at the beginning of Section~1 of \cite{Q96} that, in our setting of
discrete groups $G$, every coaction of a discrete group is coaction-nondegenerate. This
assertion was used in results of \cite{CLSV2009} that we in turn will want to use in the
proof of Theorem~\ref{thm:CNPisomorphism}. However, this assertion in \cite{Q96} depends
on \cite[Proposition~2.5]{Q94}, and a gap has recently been identified in the proof of
this result \cite{KQerratum}. The following simple lemma is well known, but hard to find
in the literature. We will use it first to show that the coactions used in
\cite{CLSV2009} are indeed coaction-nondegenerate (so the results of \cite{CLSV2009} are
not affected by the issue identified in \cite{KQerratum}), and then again in the proof of
Lemma~\ref{lemma:coactionfromgrading} below.

Recall that if $\delta : A \to A \otimes C^*(G)$ is a coaction of a discrete group, then
for each $g \in G$, we write $A_g$ for the \emph{spectral subspace} $\{a \in A :
\delta(a) = a \otimes i_G(g)\}$.

\begin{lemma}\label{lem:coactionnondegeneracy}
Let $A$ be a $C^*$-algebra and $G$ a discrete group. Suppose that $\delta : A \to A
\otimes C^*(G)$ is a coaction. Then $\delta$ is coaction-nondegenerate if and only if $A
= \clsp\bigcup_{g \in G} A_g$.
\end{lemma}
\begin{proof}
First suppose that $\delta$ is coaction-nondegenerate. Then \cite[Proposition~A.31]{EKQR}
shows that $A$ is densely spanned by its spectral subspaces. Now suppose that $A$ is
densely spanned by its spectral subspaces. Fix a typical spanning element $a \otimes
i_G(G)$ of $A \otimes C^*(G)$. Fix $\varepsilon$ and choose finitely many $g_i \in G$ and
$a_i \in A_{g_i}$ such that $\|a - \sum_i a_i\| < \varepsilon$. Then
\[\textstyle
\Big\|\sum_i \delta(a_i) (1 \otimes i_G(g_i^{-1}g)) - a \otimes i_G(g)\Big\|
    = \Big\|\Big(\sum_i a_i - a\Big) \otimes i_G(g)\Big\|
    < \varepsilon.\qedhere
\]
\end{proof}

\begin{corollary}
The coactions of $G$ on $\NT_X$ and $\NO_X$ used in \cite{CLSV2009} are
coaction-nondegenerate.
\end{corollary}
\begin{proof}
By construction (see \cite{Fowler2002}), the algebra $\NT_X$ is the closure of the span of the
elements $i_X(x) i_X(y)^*$ where $x,y \in X$. Hence $\NO_X$ is densely spanned by the
corresponding elements $j_X(x) j_X(y)^*$. The coactions of \cite{CLSV2009} are given by
$\delta(i_X(x)) = i_X(x) \otimes i_G(g)$ and $\delta(j_X(x)) = j_X(x) \otimes i_G(g)$
whenever $x \in X_g$. So each spanning element of $\NT_X$ and of $\NO_X$ belongs to a
spectral subspace for $\delta$. Hence $\NT_X$ and $\NO_X$ are spanned by their spectral
subspaces. Thus Lemma~\ref{lem:coactionnondegeneracy} shows that the coactions $\delta$
are coaction-nondegenerate.
\end{proof}

The second lemma that we need establishes that the $C^*$-algebra of the Fell bundle of
Section~\ref{sec:Fell bundle} carries a coaction of $G$ that is compatible with the gauge
coactions on $\NT_X$ and $\NO_X$.

\begin{lemma}
\label{lemma:coactionfromgrading} Let $c$ be a continuous grading of a Hausdorff \'etale
groupoid $\Gg$ by a discrete group $G$, and let $\Ee$ be a Fell bundle over $\Gg$. Let
$i_G : G \to C^*(G)$ denote the universal representation of $G$. There is a
coaction-nondegenerate coaction $\delta$ of $G$ on $C^*(\Ee,\Gg)$ satisfying
\[
 \delta(f) = f \otimes i_G(g)
\]
whenever $g\in G$ and $f \in \Gamma_c(\Gg;\Ee)$ satisfies ${\rm supp}(f) \subset
c\inv(\{g\})$.
\end{lemma}
\begin{proof}
As a vector space, $\Gamma_c(\Gg; \Ee)$ is equal to the algebraic direct sum
$\bigoplus_{g \in G} \Gamma_c(c^{-1}(g); \Ee)$. So there is a linear map $\delta :
\Gamma_c(\Gg; \Ee) \to \Gamma_c(\Gg; \Ee) \otimes C^*(G)$ such that $\delta(f) = f
\otimes i_G$ whenever $f \in \Gamma_c(c^{-1(g)}; \Ee)$. It is routine to check that this
map is continuous in the inductive-limit topology, and therefore extends to a
homomorphism $\delta : C^*(\Gg, \Ee) \to C^*(\Gg, \Ee) \otimes C^*(G)$. An elementary
calculation checks the coaction identity on $f \in \Gamma_c(c^{-1}(g); \Ee)$, which
suffices by linearity and continuity. To check that $\delta$ is coaction-nondegenerate,
observe that the spectral subspaces $C^*(\Gg, \Ee)_g$ are precisely the spaces
$\overline{\Gamma_c(c^{-1}(g)); \Ee)}$. By definition, $C^*(G, \Ee)$ is the closure of
$\Gamma_c(\Gg; \Ee)$, which is spanned by the spaces $\Gamma_c(c^{-1}(g)); \Ee)$.  It
follows that $C^*(\Gg, \Ee)$ is densely spanned by its spectral subspaces, and so
$\delta$ is coaction-nondegenerate by Lemma~\ref{lem:coactionnondegeneracy}.
\end{proof}

Recall that the Cuntz--Nica--Pimsner algebra $\NO_X$ has a quotient $\NO^r_X$ that
possesses a co-universal property described in \cite[Theorem~4.1]{CLSV2009}.

\begin{proof}[Proof of Theorem \ref{thm:CNPisomorphism}]
To show that $\Pi$ is an isomorphism, it is enough to show that the homomorphism $\Phi =
\Psi^{-1}$ of~\eqref{eqn:well-defined} factors through the quotient map
\[
\rho : C^*(\Gg,\Ee) \to C^*(\Gg|_{\partial\Omega}, \Ee)
\]
defined on $\Gamma_C(\Gg;\Ee)$ by
\[
\rho(f) = f|_{\Gg|_{\partial\Omega}}.
\]
To see this we use the co-universal property of $\NO_X^r$. Since $\Gg|_{\partial\Omega}$
is $G$-graded via $(g,\omega) \mapsto g$, Lemma~\ref{lemma:coactionfromgrading} gives a
coaction $\beta : C^*(\Gg|_{\partial\Omega}, \Ee) \to C^*(\Gg|_{\partial\Omega}, \Ee)
\otimes C^*(G)$ such that
\[
 \beta(f^S) = f^S \otimes i_G(qp\inv)\quad\text{ for all $X \in \Kk(X_p, X_q)$}.
\]
For any $x\in X_p$, we have
$$
\beta(\pi(x)) = \beta(f^x) = f^x \otimes i_G(p) = ((\pi \otimes 1) \circ \delta)(j_X(x)),
$$
where $j_X : X \to \NO_X$ is the universal representation. So $\pi$ is gauge-compatible
in the sense of \cite{CLSV2009}. We aim to apply \cite[Theorem~4.1]{CLSV2009} to $\pi$,
so we must show that $\pi_e : A \to C^*(\Gg|_{\partial\Omega}, \Ee)$ is injective. Since
the $\phi_p$ are injective, the maps $i_r : \Ll(X_p) \to \Ll(X_{pr})$ appearing in the
construction of the fibres $A_\omega$, $\omega \in \Gg^{(0)}$ in
Section~\ref{subsec:fell-fibres} are all injective. Hence the canonical map $i_\omega : A
= X_e \to X_\omega$ is injective for each unit $\omega$. In particular, for each $a \in
A$, the element $\pi_e(A) := f^a$ satisfies $f^a(\omega) = i_\omega(a) \not= 0$ for all
$\omega$, and $\pi_e$ is injective.

Now, writing $\lambda_r$ for the canonical quotient map from $\NO_X$ to $\NO_X^r$,
\cite[Theorem 4.1]{CLSV2009} yields a homomorphism
\[
 \phi : C^*(\Gg|_{\partial\Omega},\Ee) \to \NO_X^r
\]
that carries $f^S$ to $\lambda_r(j_X^{(p)}(S)$ for $S \in \Kk(X_p)$.

Fix $f\in \ker(\rho)$. Without loss of generality, assume that ${\rm supp}(f) \subset
\Gg$ is a bisection. Then $\phi(\rho(f)) = 0$ and hence $\phi(\rho(f^*f)) = 0$. So we
have $\lambda_r(\rho(\Phi(f^*f))) = 0$. But $\rho(\Phi(f^*f)) \in (\NO_X)_e$ and
$\lambda_r|_{(\NO_X)_e}$ is isometric because the reduction map for any coaction is
isometric on each spectral subspace. Hence
\[
    \|q(\Phi(f))\|^2 = \|q(\Phi(f^*f))\| = 0
\]
as required.
\end{proof}

\section{Applications}\label{sec:applications}

Takeishi \cite{T2013} has recently characterised nuclearity for $C^*$-algebras of Fell
bundles over \'etale groupoids as follows.

\begin{theorem}[{\cite[Theorem 4.1]{T2013}}]
Let $\Ee$ be a Fell bundle over an \'etale locally compact Hausdorff groupoid $\Gg$. If
$\Gg$ is amenable, then the following conditions are equivalent
\begin{itemize}
 \item[(i)] The $C^*$-algebra $C^*_r(\Ee)$ is nuclear.
 \item[(ii)] The fibre $E_x$ is nuclear for every $x\in G^{(0)}$.
 \item[(iii)] The $C^*$-algebra $C_0(\Ee|_{G^{(0)}}, G^{(0)})$ is nuclear.
\end{itemize}
\end{theorem}

For our example, the following lemma shows that (ii) holds whenever the coefficient
algebra $X_e$ of the product system $X$ is nuclear.

\begin{lemma}\label{lem:nuclear fibres}
Let $(G,P)$ be a quasi-lattice ordered group, and let $X$ be a nondegenerate finitely
aligned product system over $P$. If the coefficient algebra $X_e$ of the product system is
nuclear, then the fibres $A_\omega$, $\omega\in\Omega = \Gg^{(0)}$ are nuclear.
\end{lemma}
\begin{proof}
Fix $\omega \in \Omega$. Arguing as in Lemma~\ref{lem:fibres}, for each finite $F
\subseteq \omega$ that is closed under $\vee$, writing $p_F$ for the maximum element of
$F$ the set $B_F = \sum_{p \in F} i_{p^{-1} p_F}(\Kk(X_p))$ is a $C^*$-algebra. If $F$ is
not a singleton and $q \in F$ is minimal, then $B_{F \setminus\{q\}}$ is an ideal of
$B_F$ and the quotient $B_F/B_{F \setminus \{q\}}$ is a quotient of $i_{q^{-1}
p_F}(\Kk(X_q))$ and hence a quotient of $\Kk(X_q)$.

Each $\Kk(X_p)$ is nuclear because it is Morita equivalent to $X_e$ via $X_p$, and
nuclearity is preserved by Morita equivalence \cite[Theorem 15]{aHRW2007}. Fix a finite
$F \subseteq \omega$ and a minimal $q \in F$, and write $F' = F \setminus\{q\}$. Assume
as an inductive hypothesis that $B_{F'}$ is nuclear. Since $B_F/B_{F'}$ is a quotient of
the nuclear $C^*$-algebra $\Kk(X_q)$, it is nuclear. So $B_F$ is an extension of a
nuclear $C^*$-algebra by a nuclear $C^*$-algebra, so also nuclear \cite[Proposition
2.1.2(iv)]{RS2010}. Now $A_\omega = \varinjlim_F B_F$ is nuclear because direct limits of
nuclear $C^*$-algebras are nuclear.
\end{proof}

We therefore have the following theorem.

\begin{theorem}
Let $X$ be a nondegenerate finitely-aligned product system over a quasi-lattice ordered
group $(G,P)$, and suppose that the coefficient algebra $X_e$ is nuclear. If the groupoid
$\Gg$ of Section~\ref{sec:Fell bundle} is amenable, then $\NT_X$ and $\NO_X$ is nuclear.
If $\Gg|_{\partial\Omega}$ is amenable and the homomorphisms $\phi_p : A \to \Ll(X_p)$
implementing the left actions in $X$ are all injective, then $\NO_X$ is nuclear.
\end{theorem}
\begin{proof}
If $\Gg$ is amenable, then $C^*(\Gg, \Ee)$ is amenable by \cite[Theorem 4.1]{T2013} and
Lemma~\ref{lem:nuclear fibres}. Since $\NT_X \cong C^*(\Gg, \Ee)$ by
Theorem~\ref{thm:Toeplitzisomorphism}, we have $\NT_X$ nuclear, and then $\NO_X$ (as
defined in \cite{SY}) is nuclear because it is a quotient of $\NT_X$. If
$\Gg|_{\partial\Omega}$ is amenable then $C^*(\Gg|_{\partial\Omega}, \Ee)$ is nuclear by
\cite[Theorem 4.1]{T2013} and Lemma~\ref{lem:nuclear fibres}. If the $\phi_p$ are
injective, then Theorem~\ref{thm:CNPisomorphism} gives an isomorphism $\NO_X \cong
C^*(\Gg_{\partial\Omega}, \Ee)$, and so $\NO_X$ is nuclear.
\end{proof}

We also obtain an improvement on \cite[Corollary~4.2]{CLSV2009}. There it is proved that
$\NO_X$ and $\NO_X^r$ coincide whenever the group $G$ is amenable. But our results show
that in fact $\NO_X = \NO_X^r$ whenever $\Gg|_{\partial \Omega}$ is amenable.

\begin{proposition}
Let $X$ be a nondegenerate finitely aligned product system over a quasi-lattice ordered
group $(G,P)$, and suppose that the homomorphism $\phi_p : X_e \to \Ll(X_p)$ implementing
the left actions in $X$ are all injective. If $\Gg|_{\partial\Omega}$ is amenable, then
the quotient map $\lambda_r : \NO_X \to \NO_X^r$ is an isomorphism.
\end{proposition}
\begin{proof}
Theorem~\ref{thm:CNPisomorphism} gives an isomorphism $\Pi^{-1} :
C^*(\Gg|_{\partial\Omega}, \Ee) \to \NO_X$. Write $c : \Gg \to G$ for the continuous
cocycle $c(g,\omega) = g$. Since $\operatorname{supp}\pi(x) \subseteq \{p\} \times
\partial\Omega$ whenever $x \in X_p$, we see that $\Pi((\NO_X)_g) =
\overline{\Gamma_c(c^{-1}(g); \Ee)}$ for each $g$. In particular $\Pi^{-1}$ restricts to
an isomorphism of the closure of $\Gamma_c(c^{-1}(e); \Ee) \subseteq C^*(\Gg, \Ee)$ with
$(\NO_X)_e$. Since $c^{-1}(e) = \Gg^{(0)}$, the closure of $\Gamma_c(c^{-1}(e); \Ee)$ is
$\Gamma_0(\Gg^{(0)}; \Ee) \subseteq C^*(\Gg, \Ee)$. It is standard that restriction of
compactly supported sections to $\Gg^{(0)}$ extends to a faithful conditional expectation
$C^*_r(\Gg, \Ee) \to \Gamma_0(\Gg^{(0)}; \Ee)$. Theorem~1 of \cite{SW} implies that
$C^*(\Gg|_{\partial\Omega}, \Ee) = C^*_r(\Gg|_{\partial\Omega}, \Ee)$, so we obtain a
faithful conditional expectation $R : C^*(\Gg, \Ee) \to \Gamma_0(\Gg^{(0)}; \Ee)$
extending restriction of compactly supported sections. Lemma~1.3(a) of \cite{Q96} shows
that there is a conditional expectation $P : \NO_X \to (\NO_X)_e$ that annihilates
$(\NO_X)_g$ for $g \not= e$, and it is routine to check that $\Pi \circ P = R \circ \Pi$.
Since $\Pi$ is an isomorphism and $R$ is a faithful conditional expectation, it follows
that $P$ is a faithful conditional expectation as well. That is, the coaction $\nu$ on
$\NO_X$ such that $\delta(j_X(x)) = j_X(x) \otimes i_G(p)$ for $x \in X_p$ is a normal
coaction, and $(\NO_X, G, \nu)$ is a normal cosystem. Corollary~4.6 of \cite{CLSV2009}
shows that $\NO_X^r$ is the $C^*$-algebra appearing in the normalisation of the cosystem
$(\NO_X, G, \nu)$, and $\lambda_r$ is the normalisation homomorphism. Since this cosystem
is already normal, we conclude that $\lambda_r$ is injective.
\end{proof}

\begin{remark}
It is worth pointing out, in light of the results in this section, that it is not
uncommon for the groupoid $\Gg|_{\partial\Omega}$ of Section~\ref{sec:Fell bundle} to be
amenable, even when $G$ is not amenable. For example, $\Gg|_{\partial\Omega}$ is amenable
when $G$ is a finitely generated free group---or more generally a finitely-generated
right-angled Artin group---and $P$ its natural positive cone.
\end{remark}

\end{document}